\documentclass{article}
\usepackage{cite}

 \newcommand {\theoremstyle} [1] { }
\newenvironment{proof}{{\noindent\it\underline{Proof}}:}{\hfill$\Box$}

\newtheorem{thm}{Theorem}[section]
 
 \theoremstyle{plain}
 \newtheorem{lem}[thm]{Lemma} 
  
 \newtheorem{cor}[thm]{Corollary}
 \theoremstyle{definition}

 \theoremstyle{remark}
 \newtheorem{rem}[thm]{Remark}
 
 \newtheorem{Def}{Definition}
\usepackage[pagewise]{lineno}

\usepackage {amssymb}

\def\R{\mathbb{R}}

\usepackage {amssymb}
\usepackage{amsmath,amsfonts, color}
\usepackage{graphicx}

\makeatother

\author{Pablo Amster$^1$ and Colin Rogers$^2$}

\title{A discretisation of the Ermakov-Painlev\'e II equation:  Dirichlet and Robin-type boundary value problems}
\date{}

\begin{document}
\maketitle

\begin{center}
$^1$ Departamento de Matem\'atica, \\
Facultad de Ciencias Exactas y Naturales\\
Universidad de Buenos Aires
and 
IMAS - CONICET\\
Ciudad Universitaria, Pabell\'on I,
(1428) Buenos Aires, Argentina 
\end{center}
\begin{center}
$^2$ {School of Mathematics and Statistics}\\
   {The University of New South Wales}\\
    {Sydney, NSW 2052, Australia}
    \end{center}
\begin{center}
 pamster@dm.uba.ar -- c.rogers@unsw.edu.au
\end{center}
\bigskip

 \begin{abstract} 
 Two-point boundary value problems for a discrete Ermakov-Painlev\'e II equation are analysed by means of topological methods. In  addition, an alternative variational approach is detailed. Existence of solutions is established for
appropriate choice of parameters.


\noindent \textit{MSC-2020}: 39A27; 34B15.
 
 \end{abstract}

\section{Introduction}
Ermakov-type nonlinear coupled systems with genesis in the classical work  \cite{ve80} have diverse physical applications in both physics and continumm mechanics \cite{crws18}. These occur, `inter alia' in nonlinear optics \cite{wwhhjm68,crbmkcha10,cgjmay72,fcmlbz90,agylaspsst91,agylasps91,agylas93,crbmha12}, spiralling elliptic soliton analysis \cite{addbmdyk10}, Bose-Einstein condensate theory \cite{jaadeo11}, oceanographic warm-core eddy evolution \cite{cr89}, 2+1-dimensional magnetogasdynamics \cite{crws11} and rotating shallow water system theory \cite{crha10,ha11}.

Nonlinear coupled two-component systems of Ermakov-Ray-Reid type, as originally introduced in \cite{jr80,jrjr80}, adopt the form
\begin{equation*}
\ddot{\alpha}+w(t)\alpha=\frac{1}{\alpha^2\beta}\Phi(\beta/\alpha)\ , \quad \ddot{\beta}+w(t)\beta=\frac{1}{\alpha\beta^2}\Psi(\alpha/\beta)\ ,
\end{equation*}
wherein a dot indicates a derivative with respect to the independent variable $t$. Such systems admit a distinctive integrable of motion, namely the Ermakov invariant
\begin{equation*}
\mathrm{I}=\frac{1}{2}(\alpha\dot{\beta}-\beta\dot{\alpha})^2+\int^{\beta/\alpha}\Phi(z)dz+\int^{\alpha/\beta}\Psi(w)dw
\end{equation*}
together with concomitant nonlinear superposition principles.

In \cite{crchjr93}, 2+1-dimensional Ermakov-Ray-Reid systems were constructed. Extension of Ermakov systems to arbitrary order and dimension which preserve admittance of the characteristic invariant property were subsequently derived in \cite{crws96,wscrab96}. Therein, in particular, alignment of a 2+1-dimensional Ermakov system and an extension of the Ernst system of general relativity was shown to produce a novel integrable hybrid of the canonical 2+1-dimensional sinh-Gordon system of \cite{bkcr91,bkcr93} and a Ermakov-type system. Multi-component Ermakov systems were introduced in \cite{crws94} in the context of symmetry reduction in n-layer hydrodynamics. Sequences of Ermakov-Ray-Reid systems were shown therein to be linked via Darboux-type transformations \cite{crws02}.

In \cite{cacrurao90}, it was established that Ermakov-Ray-Reid systems possess novel underlying linear structure of a kind other than the linear-representation type admitted by solitonic systems. In addition, autonomisation of such Ermakov systems was obtained therein via a class of involutory transformations which subsequently has been extensively applied to reduce modulated physical systems to their unmodulated canonical integrable counterparts (qv \cite{crwsbm20,cr23} and literature cited therein).

The important connections between the six classical Painlev\'e equations PI-PVI and admitted symmetry reduction in modern soliton theory are 
well-esta-blished \cite{rc99,pc03}. Painlev\'e II notably has diverse physical applications not only in a solitonic context but also in boundary value problems associated with the Nernst-Planck model of ion transport \cite{wn82,mp90}. Thus, in \cite{lb64}, Painlev\'e II was derived `ab initio' in the analysis of a boundary value problem which determines the electric field distribution in a region occupied by an electrolyte. In \cite{lb1964}, this Painlev\'e reduction was applied in a description of the electrical structure of interfaces. A B\"acklund transformation admitted by Painlev\'e II as originally obtained in \cite{nl71} was applied in \cite{crabws99} to construct exact analytic representations for the electric field and ion distrubutions in electrolytic boundary value problems wherein the ratio of fluxes of negative and positive ions at the boundary adipts one of a sequence of values. In Bass \textit{et al} \cite{lbjncrws10}, iterative application of the B\"acklund transformation of \cite{nl71} admitted by Painlev\'e II was used to generate hierarchies of exact solutions to physically relevant boundary value problems for the two-ion Nernst-Planck system. Painlev\'e structure in a multi-ion context was isolated in \cite{rccrws07}. Therein, in particular, in the three-ion case, reduction was obtained to the classical Painlev\'e XXXIV equation which turns out to have a hybrid Ermakov-Painlev\'e connection via its positive solutions such as arise in the representations for ion-concentrations in a Nernst-Planck context \cite{lbjncrws10}.

Hybrid Ermakov-Painlev\'e II systems were originally derived in \cite{cr14} via a class of wave packet reductions of an n+1-dimensional Manakov-type NLS system incorporating de Broglie-Bohm potential terms. Key invariants were isolated and an algorithm presented whereby exact solution may be obtained via a procedure involving the classical Ermakov nonlinear superposition principle. Application to the analysis of transverse wave propagation in a generalised Mooney-Rivlin hyperelastic materials was shown to lead to a pair of canonical single component Ermakov-Painlev\'e II reductions. In \cite{crws16}, coupled Ermakov-Painlev\'e II systems with underlying Hamiltonian structure were shown to possess sequences of exact solutions generated via iterative action of the B\"acklund transformation admitted by the canonical Painlev\'e II equation.

The preceding attests to the established important applications of both Ermakov-type systems and the classical Painlev\'e equations to the analysis of nonlinear physical systems. Hybrid Ermakov-Painlev\'e II reductions since their introduction in \cite{cr14} have subsequently likewise proved to have diverse physical applications, notably in cold plasma physics \cite{crpc18}, Korteweg capillarity theory \cite{crpc17} and multi-ion electrodiffusion \cite{pacr15}. In \cite{cr17}, hybrid Ermakov-Painlev\'e II-IV systems were derived in an extended Ermakov-Ray-Reid context. Their admitted Ermakov invariants together with the associated canonical Painlev\'e equation were used to establish integrability properties. Novel resonant Ermakov-NLS systems which admit symmetry reduction to a hybrid Ermakov-Painlev\'e II system have been subsequently established in \cite{crws2018}.

B\"acklund transformations have been previously applied to derive discretisation in nonlinear heterogeneous elastodynamics and anisentropic gasdynamics \cite{crws97,ws98}. Thus, in \cite{crws97} the classical Bianchi nonlinear superposition principle as derived via invariance of the solitonic sine-Gordon equation under a B\"acklund transformation whereby pseudospherical surface representations may be iteratively generated \cite{crws02,cr2023} was applied to obtain integrable discretisation of characteristic equations that arise via a Monge-Amp\`ere reduction in nonlinear continuum mechanics \cite{crws82}. In \cite{ws98}, it was established that integrable discretisation of characteristic equations associated with a 1+1-dimensional anisentropic system may be derived via re-interpretation of a B\"acklund transformation admitted by the classical Tzitzeica equation. The latter has origin in the geometry of affinsph\"aren surfaces \cite{wscr94}.

In \cite{ws97}, the classical Ermakov equation was discretised in a way that preserved the inherent nonlinear superposition principle. The latter may be derived via Lie group invariance \cite{crur89}. Lie group theoretical generalisation and discretisation of Ermakov-type equations was subsequently systematically established in \cite{crwspw97}. In \cite{ah99}, an exact discretisation of the classical Ermakov equation was derived via an admitted B\"acklund transformation of the latter. Discretisation of Painlev\'e equations was obtained in a novel manner in \cite{pcemhw00} by means of B\"acklund transformations admitted by their continuous counterparts. In particular, a discretisation was obtained thereby of the classical Painlev\'e XXXIV equation which arises in symmetry reduction of certain solitonic systems. It is linked to the canonical hybrid Ermakov-Painlev\'e II equation of \cite{cr14} in a manner recorded in \cite{crws16}.

Here, discretisation is considered of the Ermakov-Painlev\'e II equation
\begin{equation}\label{disc-eq}
\Delta^2 u_{x-1} = a u_{x}^3 + bxu_{x} + \frac c{u_x^3}, \qquad x=1,\ldots, N-1
    \end{equation}
under the Dirichlet condition, with
\begin{equation}\label{dir}
u_0=D_0,\qquad u_N=D_N,    
\end{equation}
$D_0,D_N>0$, or the nonlinear Robin condition, with 
\begin{equation}\label{disc-robin}
\Delta u_0 = f_0(u_0), \qquad \Delta u_{N-1} = f_N( u_{N}).
    \end{equation}
It is assumed that 
    $f_0, f_N:(0,+\infty)\to\R$ are given continuous functions and 
    positive solutions are sought, in turn, of (\ref{disc-eq})-(\ref{dir}) and (\ref{disc-eq})-(\ref{disc-robin}). 
    It is emphasised that the type of methods for the continuous Ermakov-Painlev\'e II model applied in  previous work \cite{pacr15} 
    cannot, in general, be extended to its discretised versions. This is notably 
    so for the repulsive case $c>0$, 
    for which accordingly an alternative  method is developed.

\section {Preliminaries}

In this section, we summarize some useful facts concerning second order discrete equations
that shall be employed throughout the paper. 
Some short proofs are included for the reader's convenience.  

As usual, we may identify the set of all the functions $u:\{0, \ldots, N\} \to \R$ with $\R^{N+1}$ and 
denote $u(x):=u_x$. The first and second order discrete derivatives of $u$ shall be denoted respectively by
$$\Delta u_{x} := u_{x+1} - u_{x}\qquad x=0,\ldots, N-1,
$$
$$\Delta^2 u_{x-1} := \Delta u_{x} - \Delta u_{x-1} = u_{x+1} - 2u_{x} + u_{x-1} \qquad x=1,\ldots, N-1.
$$
It is noticed that the discretisation given by (\ref{disc-eq}) follows the standard procedure of a finite difference scheme  for the continuous Ermakov-Painlev\'e II equation studied in \cite{pacr15}, 
\begin{equation}\label{eq-cont}
    y''(z)= Ay(z)^3 + Bzy(z)  + \frac C{y(z)^3},\qquad z\in (0,1).
\end{equation}
Indeed, letting $u_x:=y\left(\frac xN\right)$ for $x=0,\ldots, N$ and taking into account that $y''\left(\frac xN\right)\simeq \frac 1{N^2} \Delta^2 u_{x-1}$, the discrete equation (\ref{disc-eq}) is obtained, with $a=\frac A{N^2}$, $b=\frac B{N^3}$ and $c =\frac C{N^2}$.

Although $u$, $\Delta u$ and $\Delta^2u$ belong to different spaces, we shall always employ the notation $\|\cdot\| _{2}$ to refer to the Euclidean norm, namely
$$\|u\|_2 = \left(\sum_{x=0}^N u_x^2\right)^{1/2},\qquad 
\|\Delta u\|_2 = \left(\sum_{x=0}^{N-1} (\Delta u_x)^2\right)^{1/2}$$
and $$
\|\Delta^2 u\|_2 = \left(\sum_{x=0}^{N-2} (\Delta^2 u_x)^2\right)^{1/2},
$$

\begin{lem}\label{parts}
    Assume that $u_0=u_N=0$ or $\Delta u_0 = \Delta u_{N-1}=0$, then    
    $$\|\Delta u\|_2^2 \le   \|u\|_2 \|\Delta^2 u\|_2.$$
\end{lem}
\begin{proof}
    The summation by parts formula 
yields 
    $$\sum_{x=1}^{N-1} \Delta^2u_{x-1}u_x = \Delta u_{N-1} u_{N-1} - \Delta u_0 u_1 + 
\sum_{x=2}^{N-1}\Delta u_{x-1}(u_{x-1}-u_x)$$
$$= \Delta u_{N-1} u_{N-1} - \Delta u_0 u_1 - \sum_{x=1}^{N-2}(\Delta u_{x})^2
$$
Under any of the two assumptions, it is immediately seen that 
$$\Delta u_{N-1} u_{N-1} - \Delta u_0 u_1= -(\Delta u_{N-1})^2 - (\Delta u_{0})^2,$$
so the result follows from the Cauchy-Schwarz inequality. 
\end{proof}

\begin{lem}\label{eigen}
    Assume that $u_0=u_N=0$ and let $\lambda_1:= 4\sin^2\left(\frac \pi{2N}\right)$, then
    \begin{enumerate}
        \item $ {\lambda_1} \|u\|^2_2\le \|\Delta u\|^2_2.$
    \item       $\|u\|_2\le  \frac 1{\lambda_1}\|\Delta^2 u\|_2.$ 

    \end{enumerate}
    
    \end{lem}
\begin{proof} The first inequality follows from the fact that $\lambda_1$ is the first eigenvalue of the problem 
$$-\Delta^2u_{x-1} =\lambda u_x\qquad u_0=u_N=0,$$
while the second one is deduced from the first one, combined with the preceding lemma. 
    \end{proof}
     \begin{lem}
         Let $g_x:\R\to \R$ be continuous and bounded for $x=1,\ldots, N-1$. 
         Then  the 
         equation 
         $$\Delta^2u_{x-1} - u_x = g_{x}(u_x) \qquad  x=1,\ldots, N-1$$
         has at least one solution, both for 
          arbitrary 
         Dirichlet conditions (\ref{dir}) and  Robin conditions 
         (\ref{disc-robin}) with  $f_0, f_N:\R\to \R$ continuous and bounded.  
     
     \end{lem}

\begin{proof}
For each $v\in \R^{N+1}$ and both conditions (\ref{dir}) and (\ref{disc-robin}),
define $T(v):=u$ as the unique solution of the linear problem
$$Lu_x:=\Delta^2u_{x-1} - u_x = g_{x}(v_x) \qquad  x=1,\ldots, N-1
$$
satisfying (\ref{dir}) or 
$$\Delta u_0 = f_0(v_0), \qquad \Delta u_{N-1} = f_N( v_{N})
$$
respectively. It is verified that $T:\R^{N-1}\to \R^{N-1}$ 
is well defined and continuous. In fact, in both cases the problem may be written in matricial form as 
$Au=B$, where $A$ is a (constant) invertible tridiagonal matrix and $B$ is a bounded vector depending 
continuously on $v$. This is  readily applied to prove that the range of the mapping $T$ is bounded, so by Brouwer's fixed point theorem the existence of at least one fixed point is deduced. 
\end{proof}

The preceding lemma allows 
to adapt to the present context the method of upper and lower solutions for a general equation 
\begin{equation}
    \label{disc-eq-gral} \Delta^2 u_{x-1} = G_x(u_x)\qquad x=1,\ldots, N-1 
\end{equation}
with $G_x$ continuous for $x=1,\ldots, N-1$.
We recall that the method was successfully extended to discrete second order problems under various boundary conditions, see e. g. \cite{AOR, AW, CO, FORP}. However, to the best of our knowledge, no results can be found in the literature for the nonlinear Robin condition (\ref{disc-robin}). For convenience, we shall assume without loss of generality 
that the continuous functions $f_0$ and $f_N$, as well as $G_x$ for $x=1,\ldots, N-1$  are 
defined on the whole line. 

\begin{Def}
We shall say that $\beta$ is an upper solution for (\ref{disc-eq-gral})-(\ref{dir}) or 
(\ref{disc-eq-gral})-(\ref{disc-robin}) if 
$$\Delta^2 \beta_{x-1} \le G_x(\beta_x)\qquad x=1,\ldots, N-1
$$
and
$$\beta_0\ge D_0, \quad \beta_N\ge D_N
$$
or 
$$\Delta \beta_0 \le f_0(\beta_0),\quad \Delta \beta_{N-1} \ge f_N(\beta_N)
$$
respectively. A lower solution $\alpha$ is defined analogously, with all the inequalities reversed. 

\end{Def}

\begin{lem} \label{sub-sup} Assume there exist $\alpha$ and $\beta$ as before such that $\alpha_x\le \beta_x$ for all $x=0,\ldots, N$. 
Then the respective problem  
    (\ref{disc-eq-gral})-(\ref{dir}) or 
(\ref{disc-eq-gral})-(\ref{disc-robin})
has at least one solution $u$ such that $\alpha_x\le u_x\le \beta_x$ for all 
 $x=0,\ldots, N$. 
 \end{lem}
\begin{proof}
    Define as usual the truncation function
    $$\mathcal T_x(u)=\max \{\min\{u, \beta_x\}, \alpha_x\} $$
    and, 
    using the previous lemma, 
    set $u$ as a solution of the problem
    $$\Delta^2 u_{x-1} - u_x = G_x(\mathcal T_x(u_x))-\mathcal T_x(u_x)\qquad x=1,\ldots, N-1
    $$
    under the condition (\ref{dir}) or
    $$\Delta u_0=f_0(\mathcal T_0(u_0)), \quad \Delta u_{N-1}=f_{N}(\mathcal T_N(u_{N}))$$
respectively. Suppose for example that 
$u-\beta$ achieves a maximum at some 
$x =1,\ldots, N-1$, then $\Delta (u-\beta)_{x-1}\le 0$, If $u_x>\beta_x$, then 
$$\Delta u_{x-1} > G_x(\beta_x) \ge \Delta \beta_{x-1}, 
$$
a contradiction. In the same way, it is verified that $u-\alpha$ cannot achieve a negative minimum at $x=1,\ldots, N-1$. In the Dirichlet case, it is also clear that $\alpha_x\le u_x\le \beta_x$ for $x=0$ and $x=N$; thus, $u$ lies between $\alpha$ and $\beta$ for all $x$ and solves the original problem. For the Robin conditions,  suppose that $u-\beta$ does not achieve a maximum at $x=1,\ldots, N-1$, then a strict absolute maximum is achieved either at $x=0$ or $x=N$. In the first case, if moreover 
$u_0 >\beta_0$, then $u_0-\beta_0 > u_1-\beta_1$, that is
$$\Delta \beta_0 \le 
f(\beta_0) =   \Delta u_0 < \Delta \beta_0,
$$
a contradiction. The proof is analogous for 
$x=N$, so we are able to conclude that $\beta\ge u$. Similarly,  it is 
verified that $\alpha\le u$, whence $u$ is a solution of the original problem.  
 \end{proof}
\section{Attractive case}

Throughout this section, we shall assume that the singularity is attractive, that is, $c<0$. 
Applying the lemmas of the previous section, we shall establish our results. 

\begin{thm}\label{attract-1}
   Let $a> 0 >c$. Then (\ref{disc-eq})-(\ref{dir}) admits at least one positive solution for arbitrary $D_0,D_N>0$ and 
   (\ref{disc-eq})-(\ref{disc-robin}) admits at least one positive solution, provided that 
   $$f_0(\alpha), f_N(\beta) \le 0\le f_N(\alpha), f_0(\beta)$$      
for some $\alpha >0$ small enough and some $\beta >\alpha$ large enough. 
   Furthermore, let $$M:= -\frac 9{N-1}\left(\frac{a^2c}4\right)^{1/3}  > 0.$$  
Then there are no other positive solutions  if
   \begin{equation}\label{uniq}
   b + M > - \frac 4{N-1} \sin^2\left(\frac \pi{2N}\right)  
   \end{equation}
      for the Dirichlet condition
and
\begin{equation}\label{uniq-rob}
b + M > 0    
\end{equation}
in the Robin case, provided that $f_0$ is nondecreasing and $f_N$ is nonincreasing.  
\end{thm}

\begin{proof}
    Define $G_x(t):=  a t^3 + bxt + \frac c{t^3}$, then $G(\alpha) < 0 < G(\beta)$ for   
    $\alpha>0$ sufficiently small and $\beta>0$ sufficiently large. 
    In particular, $\alpha$ and $\beta$ may be chosen in such a way that all the conditions of Lemma \ref{sub-sup} are satisfied, so the existence of at least one solution follows. 
Now assume that $u$ and $v$ are positive solutions and set $w:=u-v$, then a simple computation shows that 
$$\sum_{x=1}^{N-1} \Delta^2 w_{x-1}w_{x}  
= \sum_{x=1}^{N-1} [G_x(u_x)-G_x(v_x)](u_x-v_x)\ge (N-1)M\sum_{x=1}^{N-1} w_x^2 +  
\sum_{x=1}^{N-1}xbw_x^2,
$$
whence
$$-\|\Delta w\|_2^2 
+ w_N\Delta w_{N-1} - w_0\Delta w_0 \ge 
(N-1)M\sum_{x=1}^{N-1} w_x^2 +  
\sum_{x=1}^{N-1}xbw_x^2
$$
$$   \ge (N-1)(M-b^-)\sum_{x=1}^{N-1} w_x^2,  
$$
where, as usual, $b^-=-\min\{ b, 0\}$. 

On the one hand, 
in the Dirichlet case, it is seen that $w_0=w_N=0$, so the previous inequality, combined with the assumption (\ref{uniq}) yields
$$-\|\Delta w\|_2^2 \ge (N-1)(M-b^-)\|w\|_2^2 > -\lambda_1\|w\|_2^2,$$
 and the proof follows from Lemma \ref{eigen}.

On the other hand, under the Robin 
condition, the monotonicity assumptions on $f_0$ and $f_N$ imply 
$$w_N\Delta w_{N-1} \le 0 \le w_0\Delta w_0;
$$
thus, the assumption  $M-b^->0$ directly implies 
that 
$$\|\Delta w\|^2 = \sum_{x=1}^{N-1} w_x^2=0.$$
We conclude that $w$ is constant and $w_x=0$ for $x=1,\ldots, N-1$ which, in turn, implies $w_0=w_N=0$. 
    \end{proof}

 \begin{rem}
     It is of interest to compare the previous result with the corresponding one for the continuous equation (\ref{eq-cont}); particularly, the uniqueness conditions (\ref{uniq}) and (\ref{uniq-rob}) which, letting $N\to \infty$, yield
     $$B - 9\left(\frac{A^2C}4\right)^{1/3} \ge -\pi^2
     $$
    and
    $$B\ge 0$$
    respectively. The latter two inequalities can be easily obtained from (\ref{eq-cont}) by direct computation and, in the Dirichlet case,  the condition given in \cite{pacr15} is thus improved.
     
 \end{rem}

Next, we shall consider the case $a\le 0$. In the previous paper \cite{pacr15}, a somewhat sharp sufficient condition was obtained by considering the solutions of the autonomous problem 
$$u''(x)= au(x)^3 - b^-u(x) + \frac c{u(x)^3}, \quad 0<x<1  
$$
as upper solutions of the original continuous version of (\ref{disc-eq}). 
However, the bounds given in \cite{pacr15} strongly 
relied  on the fact that 
the autonomous equation can 
be integrated when multiplied by $u'(x)$,
 a procedure
that cannot be imitated here, due to the failure of the chain rule. 
In contrast with that, the discrete problem has an advantage with respect to the continuous one, for which a  constant upper solution $\beta>0$ does not exist  when $a\le 0$, since the inequality 
$$bx\beta\ge -\left(a\beta^3  + \frac c{\beta^3}
\right)
>0 $$ 
needs to be satisfied for all 
$x\in (0,1)$. 
The following result provides a sufficient condition for the existence of a constant upper 
solution for  (\ref{disc-eq}).

\begin{thm}\label{th2}
Assume $c<0<b$, then: 

\begin{enumerate}
    \item For $a=0$, the Dirichlet problem (\ref{disc-eq})-(\ref{dir}) 
    has a unique  positive solution. If furthermore 
               $$f_0(\alpha) \le 0\le f_N(\alpha),\qquad 
    f_0(\beta)\ge 0\ge f_N(\beta)
    $$
for some sufficiently small $\alpha>0$ and some sufficiently large $\beta>\alpha$, then  the nonlinear Robin problem (\ref{disc-eq})-(\ref{disc-robin}) has also a positive solution, which is unique when $f_0$ is nondecreasing and $f_N$ is nonincreasing.

    \item For $a<0$,  assume that
    \begin{equation}\label{beta-cond}
    4b^3 \ge -27ca^2,
    \end{equation}        
and 
    define 
\begin{equation}\label{beta}
    \beta(b) = \sqrt{\frac{-2b}{3a}}.
    \end{equation}
    Then (\ref{disc-eq}) has at least one positive solution satisfying (\ref{dir}) or (\ref{disc-robin}), provided that 
    $$\beta(b)\ge D_0, D_N
    $$
    and
    $$f_0(\alpha) \le 0\le f_N(\alpha),\qquad 
    f_0(\beta(b))\ge 0\ge f_N(\beta(b))
    $$
    for some sufficiently small $\alpha>0$,
    respectively. 
    \end{enumerate}
\end{thm}
\begin{proof} The proof for the case $a=0$ follows exactly as in the preceding result. For $a<0$, let $G_x$ be defined as in the previous proof, then clearly $G_x\ge G_1$, so it suffices to verify that $G_1(\beta(b))\ge 0$. In turn, this is equivalent to inequality (\ref{beta-cond}), 
    so the result follows.
\end{proof}

 \begin{rem}
     According with the preceding comment regarding the continuous model (\ref{eq-cont}), it is observed, in the second case,  that (\ref{beta-cond}) cannot be satisfied when $N$ is large.   
     
 \end{rem}

It is observed, in the latter proof, that  $\beta(b)$ 
is computed as the (unique) positive value in which the absolute maximum of the 
function $\varphi(z):= az^6 + bz^4$ with $z\ge 0$ is achieved. This shows that, if the inequality in (\ref{beta-cond}) is strict, then the conditions may be relaxed by considering the maximal interval $\mathcal I_c\subset (0,+\infty)$ such that $\varphi(z)\ge -c$ for $z\in I_c$. 
With this in mind, let us set the notation $M_c$ for the upper endpoint of $\mathcal I_c$. 

\begin{cor} For $a,c < 0 < b$, assume that  (\ref{beta-cond}) holds and let 
$\mathcal I_c$ and $M_c$ be defined as before. Then
\begin{enumerate}
    \item Problem (\ref{disc-eq})-(\ref{dir}) has at least one positive solution for arbitrary $D_0, D_N\in (0,M_c]$. 
    \item Problem (\ref{disc-eq})-(\ref{disc-robin}) has at least one positive solution, provided that 
    $$f_0(\alpha) \le 0\le f_N(\alpha),\qquad 
    f_0(\beta)\ge 0\ge f_N(\beta)
    $$
    for some sufficiently small $\alpha>0$ and some $\beta\in \mathcal I_c$. 
    
\end{enumerate} 
    
\end{cor}

The next corollary is directly deduced from the fact that, 
as $b\to +\infty$, the value
$\beta(b)$ tends to $+\infty$ and the lower endpoint of $\mathcal I_c$ tends to $0$. 

\begin{cor} Assume that $a, c< 0$, then:
\begin{enumerate}
    \item Given arbitrary $D_0, D_N>0$, there exists $b^*>0$ such that problem (\ref{disc-eq})-(\ref{dir}) has at least one positive solution if $b\ge b^*$.     
    \item There exists $b^*>0$ such that problem (\ref{disc-eq})-(\ref{dir}) has at least one positive solution if $b\ge b^*$, provided that 
    $$
    f_0(\beta)\ge 0\ge f_N(\beta)$$
    for some $\beta >0$  and
    $$f_0(\alpha_n) \le 0\le f_N(\alpha_n)
    $$
       for some sequence $\alpha_n\to 0^+$.
        \end{enumerate}
\end{cor}
In the same spirit of the latter corollary, it is seen that solutions exist when the negative constant $c$ is close enough to $0$:

\begin{cor} Assume that $a< 0 < b$ and define $\beta(b)$ as before. Then:
\begin{enumerate}
    \item There exists $c_*>0$ such that (\ref{disc-eq})-(\ref{dir}) has at least one positive solution for $-c_*<c<0$, provided that $\beta(b)\ge D_0, D_N$. 
    \item There exists $c_*>0$ such that (\ref{disc-eq})-(\ref{disc-robin}) has at least one positive solution for $-c_*<c<0$, provided that 
    \begin{enumerate}
        \item $f_0(\beta(b)) \ge 0 \ge f_N(\beta(b))$.
        \item $f_0(\alpha_n)\le 0\le f_N(\alpha_n)$ for some sequence $\alpha_n\to 0^+$. 
        
    \end{enumerate}
\end{enumerate}
\end{cor}

\subsection{The homogeneous Dirichlet problem}

In this section, we shall consider the problem (\ref{disc-eq}) under the homogeneous condition
\begin{equation}
    \label{dir0} u_0=u_N =0.
\end{equation}
In contrast with the continuous case, 
here the singularity at the boundary does not imply that the derivatives are unbounded. In fact, if a sequence   of solutions  is bounded, then so is the sequence of its derivatives. This suggests  us to look for a solution of (\ref{disc-eq})-(\ref{dir0}) as the limit of a sequence 
of positive solutions of  (\ref{disc-eq})-(\ref{dir})
with Dirichlet conditions converging to $0$. 
Such a procedure yields the following result. 

\begin{thm} \label{attract-hom} Problem (\ref{disc-eq})-(\ref{dir0}) has at least one solution with $u_x >0$ for $x=1,\ldots, N-1$, provided that one of the following conditions holds:
    \begin{enumerate}
        \item $a > 0>c$.
        \item $a=0$ and $b>0>c$. 
        \item $a,c<0$ and $4b^3 > -27ca^2$.  
    \end{enumerate}
    Furthermore, the solution is unique in the second case, and also in the first case under the assumption (\ref{uniq}).
\end{thm}

\begin{proof} Fix an arbitrary sequence $r_k\searrow 0$ and a constant $\beta>0$ such 
that $G_x(\beta)\ge 0$ for $x=1,\ldots, N-1$. 
According to the results in the preceding section, $\beta$ 
serves as an upper solution of the Dirichlet problem when 
$D_0, D_N\le \beta$. 
Thus, we may assume that  
$r_k$ is a  lower solution 
with $r_k< \beta$ for all $k$ 
and define $u{(k)}$ as a solution of (\ref{disc-eq}) between $r_k$ and $\beta$ satisfying
$u(k)_0=u(k)_N = r_k$.  
Since $u(k)_x\le \beta$ for all $x=0,\ldots,N$, taking a subsequence we may assume that $\{u(k)\}$ converges  to some $u$ such that $u_0=u_N=0$. Furthermore, the identity 
$$
 \Delta^2 u(k)_{x-1} = a u(k)_{x}^3 + bxu(k)_{x} + \frac c{u(k)_x^3} \qquad x=1,\ldots, N-1
$$
combined with the fact that the left-hand side term is bounded imply that $u_x>0$ and
$u$ verifies (\ref{disc-eq}) for $x=1,\ldots, N-1$. 
Uniqueness is deduced exactly as before. 
\end{proof}
 
  \section{Repulsive case}

Here, we deal with the case $c>0$. 
We shall consider firstly the Dirichlet 
problem. 

\begin{thm}
    Assume $c>0>a$, then 
     (\ref{disc-eq})-(\ref{dir}) with $D_0, D_N\ge 0$ has at least one solution $u$ with $u_x>0$ for $x=1,\ldots, N-1$. 
\end{thm}
\begin{proof}
 Let us firstly observe that the problem  may be also written as
 $$u_{x+1} + u_{x-1} = a u_{x}^3 + (2+bx)u_{x} + \frac c{u_x^3} \qquad x=1,\ldots, N-1,
 $$
 where it is already assumed that $u_0=D_0$ and $u_N=D_N$. 
 This shows that $u$ is a (positive) solution if and only if $P_x(u)=0$ for all $x=1,\ldots, N-1$, where
$$P_x(u)= au_x^6 + (2+bx)u_x^4 -(u_{x-1} + u_{x+1}) u_x^3 + c.$$
Set 
$$\Omega:=\{ u\in \R^{N-1}: 0<u_x< R, \; x=1,\ldots, N-1 \}$$
for some $R>0$ to be defined, and  the 
homotopy $H:\overline\Omega\times [0,1]\to \R^{N-1}$ given by
$$H(u,\lambda)_x:= 
au_x^6 + \lambda \left[(2+bx)u_x^4 -(u_{x-1} + u_{x+1}) u_x^3\right] + c.
$$
It is observed that, if $u_x=0$ for some $x$, then $H_\lambda (u)_x:=H(u,\lambda)_x=c\ne 0$. Furthermore, taking $R$ sufficiently large it is also clear that also
$H_\lambda (u)_x\ne 0$ when $u_x=R$; in other words, we have verified that 
$H_\lambda$ does not vanish on $\partial \Omega$ for arbitrary $\lambda\in [0,1]$. From the homotopy invariance of the Brouwer degree, we deduce that 
$$\deg(H_1,\Omega,0)= \deg(H_0,\Omega,0)= (-1)^{N-1},
$$
where the latter equality follows from the 
fact that $H_0(u)_x=au_x^6+c$, which has a 
unique simple positive root $u_*:= \sqrt[6]{\frac{-c}a}$ and the sign of the Jacobian determinant of $H_0$ at $(u_*,\ldots, u_*)$ coincides with that of $a^{N-1}$. 
Hence, the existence property of the Brouwer degree implies that 
$H_1$ has at least one root in $\Omega$, 
which corresponds to a solution of the problem.
\end{proof}

\begin{rem} \label{polynomials}
    \begin{enumerate}
        \item The previous proof is also valid when $c<0<a$; thus, the existence part of 
        Theorem \ref{attract-1} is retrieved, including the homogeneous case treated in Theorem \ref{attract-hom}. 
        In both situations, the fact that $ac<0$ allows the direct  use the Poincar\'e-Miranda theorem instead of the Brouwer degree. 
        Indeed, it suffices to observe, for $u\in \overline \Omega$, that 
        $$P_x(u)=c\qquad \hbox{if $u_x=0$}$$
        and 
        $$P_x(u)\sim aR^6\qquad \hbox{if $u_x=R$},$$
which shows that $P_x$ changes sign at the corresponding faces of $\Omega$. 
However, the topological degree is more general and  may constitute an useful tool when searching for multiple solutions. 

        \item It is worth recalling, 
        for the non-homogeneous continuous case (\ref{eq-cont}), that the assumption $A<0<C$ yields in fact the existence of infinitely many solutions. More precisely, for all $k\in \mathbb N$ sufficiently large there exist at least two solutions with exactly $k$ nodal regions with respect to the segment joining the boundary values. 
        Such a result cannot be reproduced in the discrete model for obvious reasons, although it is expected that the number of solutions increases as $N$ gets larger. 
        As a toy model, let  $N=2$, 
        then the problem reduces to find the zeros of the polynomial 
        $$P_1(u_1)=au_1^6 + (2+b)u_1^4 -(D_{0} + D_{2}) u_1^3 + c,$$
which trivially verifies 
$$P_1(0)= c>0,\qquad P_1(u_1)\to -\infty\quad\hbox{as $u_1\to+\infty$}.$$
By computing its derivative (or simply by the Descartes rule), it is easy to conclude that there exist at most $3$ positive solutions; moreover, 
the solution 
is unique if and only if 
$$b+2 < \frac 94 \left(\frac {-a(D_{0} + D_{2})^2}2 \right)^{1/3}.$$

    \end{enumerate}
\end{rem}

Next, we proceed with the Robin case:

\begin{thm}\label{rob-rep}
Assume $a<0<c$ and 
\begin{enumerate}
    \item $f_0(\eta) \le 0 \le f_N(\eta)$ for some $\eta>0$ small enough.
    \item There exists a sequence $R_n\to+\infty$ such that 
    $$\liminf_{n\to\infty} \frac {f_0(R_n)}{R_n} > -1, \qquad   \limsup_{n\to\infty} \frac {f_N(R_n)}{R_n} < 1. 
    $$
 \end{enumerate}
   Then (\ref{disc-eq})-(\ref{disc-robin}) has at least one positive solution. 
\end{thm}
  \begin{proof}
  We shall follow the outline of the previous proof, now     taking into account the Robin condition, which motivates to define
  $$P_0(u):=f_0(u_0) + u_0 - u_1,\qquad P_N(u):=-f_N(u_N) + u_N - u_{N-1}.    
  $$
  Within this context, the polynomials $P_x$ for $x=1,\ldots x_{N-1}$ are defined as before for arbitrary $u\in (0,+\infty)^{N+1}$.   
  It is noticed that the   homotopy of the preceding proof may  be extended  by setting 
$$H_\lambda(u)_0:=  \lambda f_0(u_0) + u_0 - u_1,\qquad 
H_\lambda(u)_N:=  -\lambda f_N(u_N) + u_N - u_{N-1}.$$
According to the assumptions, we may 
fix $\varepsilon>0$ and $n_0$ such that
$$f_0(R_n) + R_n  > \varepsilon R_n, \qquad f_N(R_n) - R_n < -\varepsilon R_n
$$
for all $n\ge n_0$. 
Next, 
fix $R:= R_n$ for some sufficiently large $n\ge n_0$ such that if $H_\lambda (u)_x =0$ for all $x=1,\ldots, N-1$ with 
$0< u_y \le R$ for all $y=0,\ldots, N$, then $u_x \le \varepsilon R$ for $x=1,\ldots, N-1$. 
Suppose now that 
$H_\lambda(u) = 0$ with $0<u_x\le R$ for all $x$ and $u_0=R$, then 
$$\lambda f_0(R) + R = u_1\le \varepsilon R.  
$$
This implies $f_0(R)<0$ and, consequently, 
$$\lambda f_0(R) + R \ge f_0(R) + R >\varepsilon R,$$
 a contradiction. In the same way, it is verified that, 
 if $H_\lambda(u) = 0$ with $0<u_x\le R$, then $u_N<R$. Now fix 
 $\eta>0$ such that 
 if $0<u_x\le R$ for all $x=0,\ldots N$ and 
 $u_j\le\eta$ for some $j=1,\ldots, N-1$, then 
 $H_\lambda(u)_j\ne 0$. Suppose that $H_\lambda(u)=0$ with $\eta \le u_x<R$ for all $x$ and $u_0=\eta$, then 
 $$\lambda f_0(\eta) + \eta =u_1 >\eta.$$
 From the hypothesis, $\eta$ may be chosen in such a way that $f_0(\eta)\le 0$, which yields a contradiction. In the same way, we deduce that $u_N>\eta$, whence the homotopy does not vanish on $\partial\Omega$, where 
 $$\Omega:= \{ u\in \R^{N+1}: \eta<u_x<R,\quad x=0,\ldots, N\}.
 $$
 Again, this implies 
 $$\deg(H_1,\Omega,0)=\deg(H_0,\Omega,0)=(-1)^{N+1}$$
  because the Jacobian matrix of $H_0$ at the unique root $(u_*,\ldots,u_*)$ is now given by
  $$\left(\begin{array}{ccccc}
      1 & -1 & 0 & \ldots  & 0 \\
      0 & 6au_*^5 & 0 &\ldots & 0\\
      0 & \ldots  &\ldots & \ldots & 0\\ 
      0 & 0 &\ldots & 6au_*^5 & 0\\
      0 & 0 & \ldots & -1 & 1
  \end{array}
  \right). 
  $$
  \end{proof}

Next,  
we shall focus on the case  $a,c>0$. 
In the Dirichlet case, it is immediately seen that solutions cannot exist when $c$ is sufficiently large. 
The same happens in the  Robin case, provided that $f_0$ and $f_N$ satisfy some appropriate growth conditions. 
The following result shows  that solutions exist in the opposite situation. 

\begin{thm}
Let $a>0$, then: 

\begin{enumerate}
    \item There exists $c_*>0$ such that  problem (\ref{disc-eq})-(\ref{dir}) has at least one positive solution for $0<c<c_*$, in the following cases:
\begin{enumerate}
    \item $D_0>0$ or $D_N>0$. 
    \item $D_0=D_N=0$ and $b < -\frac 2{N-1}$. 
\end{enumerate} 
\item There exists $c_*>0$ such that   problem (\ref{disc-eq})-(\ref{disc-robin}) has at least one positive solution for $0<c<c_*$, provided that 
    $$ {f_0(r_0)} + r_0 \le 0 \le {f_N(r_N)} - r_N 
    $$ 
    for some $r_0, r_N>0$ 
    and $$\liminf_{n\to\infty} \frac {f_0(R_n)}{R_n} > -1, \qquad   \limsup_{n\to\infty} \frac {f_N(R_n)}{R_n} < 1$$
    for some sequence $R_n\to+\infty$. 
    
\end{enumerate}

 \end{thm}
\begin{proof} Consider firstly the Dirichlet case with $D_0>0$ and define the function 
$$Q_x(u):= P_x(u) - c = u_x^3[au_x^3 + (2+bx)u_x -(u_{x-1} + u_{x+1})].$$
    Next, fix $\varepsilon_1>0$ such that 
    $$at^3 + (2+b)t  < D_0\qquad 0<t\le \varepsilon_1$$
    and, inductively, we may set $\varepsilon_2,\ldots, \varepsilon_{N-1}>0$
    such that 
    $$at^3 + (2+bx)t  < \varepsilon_{x-1}\qquad 0<t\le \varepsilon_{x}.
    $$
    Also, we may fix $R>0$ such that $aR^3 + (2+bx)R > 2R$ for $x=1,\ldots, N-1$ and define the open set 
    $$\Omega:= (\varepsilon_1,R)\times\ldots\times (\varepsilon_{N-1},R).$$
    It follows from the previous choice of the values $\varepsilon_x$ and $R$ 
    that if $u\in\partial \Omega$, then there exists $x$ such that either 
    $u_x=R$ and $Q_x(u)>0$, or $u_x=\varepsilon_x$ and $Q_x(u)<0$. Thus, picking an arbitrary point $v\in \Omega$, in both situations it is deduced that 
    $\lambda Q_x(u) + (1-\lambda)(u_x - v_x)\ne 0$ for $\lambda\in [0,1]$. In other words, the homotopy 
    $H_\lambda:= \lambda Q +(1-\lambda)(I-v)$ does not vanish, where $I$ denotes the identity map and, in consequence,
    $$\deg(Q,\Omega,0)= \deg(I-v,\Omega,0)=1.$$
    Because the degree is continuous, we conclude that 
     $$\deg(P ,\Omega,0)= \deg(Q+ (c,\ldots,c), \Omega,0)=1$$
when  $c$ is sufficiently small. 
The proof is analogous when $D_0=0<D_N$, now starting from an appropriate value $\varepsilon_{N-1}>0$
 such that 
$$at^3 + (2+b(N-1))t  < D_N\qquad 0<t\le \varepsilon_{N-1}
    $$
    and, inductively, choosing  values $\varepsilon_{N-2},\ldots, \varepsilon_1>0$ such that 
    $$at^3 + (2+bx)t  < \varepsilon_{x+1}\qquad 0<t\le \varepsilon_{x}.$$
Finally, if $D_0=D_N=0$ and $b < -\frac 2{N-1}$, it suffices to fix $\varepsilon_{N-1}>0$ such that 
$$at^3 + (2+b(N-1))t  < 0\qquad 0<t\le \varepsilon_{N-1}
    $$
    and proceed as before. 

    Concerning the Robin case, set $\varepsilon_0:=r_0$,  $\varepsilon_N:=r_N$ and $\varepsilon_x$ as before for $x=1,\ldots, N-1$. Then, fix a value $R>0$ as in the  proof of Theorem 
    \ref{rob-rep} and set 
    $$\tilde\Omega:= (\varepsilon_0,R)\times\ldots\times (\varepsilon_{N},R)$$
    $$\tilde Q(u);= (f_0(u_0)+ u_0-u_1, Q(u_1,\ldots,u_{N-1}), -f_N(u_N)+ u_N-u_{N-1}). $$
    Again, it is deduced that $\deg(\tilde Q, \tilde \Omega,0)=1$, and the proof follows from the fact that 
    $P  = \tilde Q + (0,c,\ldots, c, 0)$.  
\end{proof}

As a  conclusion of  this section, let us
emphasize that the
  case  $a<0<c$ can be 
analysed  by a more careful study of the associated polynomial system and, as  
mentioned in Remark \ref{polynomials}, the non-homogeneous Dirichlet problem for (\ref{eq-cont}) has infinitely many solutions. 
In contrast, the discretization involves a system of $N-1$ polynomials with $N-1$ variables and, as is well known, the number of (nondegenerate) positive solutions is always finite. A celebrated -although unsharp- upper bound of this number was given by Khovanskii in \cite{kho} and improved in later works; in our case, the system can be reduced to a single equation, so an upper bound is readily obtained from Descartes' rule of signs.
However, no general results for lower bounds seem to exist in the literature. 
Such bounds  would be of interest, since it is expected that the number of solutions tends to infinity as $N\to\infty$. 
Concerning the case $a,c>0$, existence of solutions of (\ref{disc-eq})-(\ref{dir}) is guaranteed when the associated polynomial system has at least one positive root but, again, the literature on 
this topic is scarce. 
Some recent results have been obtained (see e.g. 
\cite{wang}), 
but the conditions do not apply to our  specific case.

  \section{The variational approach}

In this section, we shall briefly describe the main aspects of the variational method, adapted to the present boundary value problems. 
This aims, on the one hand, to give alternative proofs of some of the preceding results and, on the other hand, to propose a useful tool in order to tackle the multiplicity problem suggested in the previous section, with the help of  linking-type theorems. 
This goal  shall be pursued in a forthcoming paper.

For the Dirichlet case, define the functional 
$\mathcal I:(0,+\infty)^{N-1}\to \R$ given by 
\begin{equation}
\label{functional}
\mathcal I(u)= \frac 12\sum_{x=1}^N (\Delta u_{x-1})^2+\frac a4 \sum_{x=1}^{N-1} u_x^4 
+ \frac b2 \sum_{x=1}^{N-1} xu_x^2 - 
 \frac c2 \sum_{x=1}^{N-1} \frac 1{u_x^2},  
\end{equation}
where it is assumed that $u_0=D_0$ and $u_N=D_N$. 
 It is straightforwardly verified that the critical points of $\mathcal I$ coincide with the positive solutions of (\ref{disc-eq})-(\ref{dir}). Next, for $x=1,\ldots, N-1$ compute
$$\frac {\partial \mathcal I}{\partial u_x}= 2u_x - (u_{x-1} + u_{x+1}) 
+ au_x^3 + bx u_x + \frac c{u_x^3}.
$$
This shows that $u$ is a critical point of $\mathcal I$ if and only if $P_x(u)=0$ for all $x=1,\ldots, N-1$, where $P_x$ are the polynomials defined in the previous section.

Analogously for condition (\ref{disc-robin}), 
the functional $\mathcal J:(0,+\infty)^{N+1}\to \R$ defined by 
  $$\mathcal J(u):= \mathcal I(u) + F_0(u_0) - F_N(u_N),$$
  with $\mathcal I$ defined as in (\ref{functional}) and 
  $$F_0(s):=\int_0^s f_0(t)\, dt, \qquad F_N(s):=\int_0^s f_N(t)\, dt.
  $$
  Here, if $u$ is a critical point of 
  $\mathcal J$, then  taking all the directional derivatives $\frac{\partial \mathcal J}{\partial \nu}(u)$
with $\nu_0=\nu_N=0$ it is deduced 
that $u$ satisfies (\ref{disc-eq}).
Next, for arbitrary $\nu\ne 0$, summation by parts yields
$$0= \frac{\partial \mathcal J}{\partial \nu}(u) =
f_0(u_0)\nu_0 - f_N(u_N)\nu_N +\Delta u_{N-1}\nu_N -\Delta u_0\nu_0,
$$
so the Robin condition is 
verified by 
taking firstly $\nu_0=0\ne \nu_N$ and next $\nu_0\ne 0= \nu_N$.
In order to verify that the  critical 
points of $\mathcal J$ coincide with the solutions of the problem, let us simply observe that 
$$\frac{\partial \mathcal J}{\partial u_x} = \frac{\partial \mathcal I}{\partial u_x} \qquad x=1,\ldots, N-1
$$
and
$$\frac{\partial \mathcal J}{\partial u_0}= 
f_0(u_0) + u_0 - u_1,\qquad \frac{\partial \mathcal J}{\partial u_N} = -f_N(u_N) + u_N - u_{N-1}.
$$

It is observed that the 
variational formulation provides an immediate resolution of the cases in which $a$ and $c$ have different signs, since a coercivity argument can be applied. For the Dirichlet case  
assume, for example, that $a>0>c$, then $\mathcal I(u)\to+\infty$ as some coordinate of 
$u$ tends to $0$ or to $+\infty$ which, in turn, implies that $\mathcal I$ achieves an absolute minimum on $(0,+\infty)^{N-1}$. A similar conclusion follows when $a<0<c$ and also for the Robin case, under appropriate assumptions on $f_0$ and $f_N$.

\section{Conclusion}

Here, a discretisation of the Ermakov-Painlev\'e II equation subject to certain boundary conditions has involved the application of both topological and 
variational approaches. The extension of the latter to a range of discretisation problems is the subject of current investigation. In addition, potential application of the B\"acklund transformation for Painlev\'e XXXIV equation and its link to the Ermakov-Painlev\'e II equation is to be investigated in connection with discretisation of the latter with appropriate concomitant boundary conditions.

\end{document}